\documentclass[12pt,a4paper]{article}
\usepackage{amsfonts,amsthm,amsmath,amssymb,amscd}
\usepackage{fancybox}
\usepackage[latin1]{inputenc}
\usepackage[usenames]{color}
\usepackage{graphicx}
\usepackage{enumerate}
\usepackage{float}
\usepackage{epsfig,amssymb} 
\usepackage{subcaption}
\usepackage{tikz}
\usetikzlibrary{calc,intersections,decorations.markings}
\allowdisplaybreaks
\usepackage{url}

\newtheorem{theorem}{Theorem}[section]

\newtheorem{rem}[theorem]{Remark}

\begin{document}
	
\title{\textbf{Red Blood Cells as Elastic Surfaces}}

\author{E. Aulisa, S. Fields, M. Toda
}
\date{May 2025}

\maketitle

\begin{abstract}
We study red blood cells using the Helfrich-Canham functional: due to their lipid bilayer structure, RBCs are naturally modeled using the theory of elastic surfaces. In this study, we demonstrate that Cassinian ovals, except for the limiting case of the round sphere, do not solve the shape equation. We further discuss conditions under which they may serve as effective approximations.
\end{abstract}
\noindent{\emph{Keywords: generalized Willmore surface, Helfrich surface, shape equation, red blood cell, closed elastic surfaces} }\\
\noindent{\emph{Mathematics Subject Classification 2020: 35G20, 35A05, 49S05} }

\section{Introduction}
Over the past two decades, there has been growing interest in identifying an optimal profile curve for modeling red blood cells (RBCs), also known as \textit{erythrocytes} in medical literature. As the most abundant blood cells, RBCs play a vital role in oxygen transport through the circulatory system. Although their shape has been studied for centuries, accurate modeling requires advanced mathematics and validation against biophysical data obtained through modern instrumentation.

Due to their lipid bilayer structure, RBCs are naturally modeled using the theory of elastic surfaces developed by Helfrich, Canham, and Evans in the 1970s~\cite{HuOuYang, Helfrich}. In this framework, RBCs minimize the so-called Helfrich-Canham energy, a curvature-dependent functional involving mean curvature, Gaussian curvature, and physical constants.

A lipid bilayer is a thin membrane of amphipathic molecules with hydrophilic and hydrophobic ends, forming the structural basis of cell membranes and organelles. These bilayers define cellular boundaries and enable ion and electron transport essential to cell function.

In~\cite{Mladenov}, RBC profiles are modeled using Cassinian ovals, with estimates provided for geometric and energetic parameters. Similarly,~\cite{Liu} derives RBC shapes via minimization of total elastic surface energy. While Cassinian ovals have been proposed as highly accurate descriptors of RBC profiles~\cite{Mladenov}, their compatibility with the Helfrich-Canham shape equation has not been rigorously established. In this study, we demonstrate that Cassinian ovals, except for the limiting case of the round sphere, do not solve the shape equation.

\begin{rem}    
The spontaneous curvature $c_0$ is assumed to be constant in most circumstances, but it is not physically true that spontaneous curvature will always be constant. In practice, there are many variable chemical and environmental factors (e.g. temperature, shear) in lipid bilayers or cell membranes which make $c_0$ variable. However, since the RBC has a rather homogeneous structure and it simplifies our first variation calculations, we justifiably assume $c_0$ is constant. See \cite{Deuling} for more information.
\end{rem}

\section{First Variation}

To model the geometry of the red blood cell (RBC) membrane, we consider the following surface energy functional:
\[
\Theta(M)=\iint_M \beta(2H - c_0)^2 \, dS + \lambda \int dA + \Delta P \int dV,
\]
commonly referred to as the \textit{Helfrich-Canham functional}. In this expression:
\begin{itemize}
    \item $\beta > 0$ is the \textit{bending rigidity} of the membrane;
    \item $c_0 \in \mathbb{R}$ is the \textit{spontaneous curvature}, as introduced in \cite{Helfrich}, \cite{HuOuYang}, and other foundational works;
    \item $H = -\frac{k_1 + k_2}{2}$ is the \textit{mean curvature} (also known as \textit{Germain curvature}), where $k_1$ and $k_2$ are the principal curvatures;
    \item $\lambda$ and $\Delta P$ are \textit{Lagrange multipliers} enforcing constraints of constant surface area and enclosed volume, respectively. These correspond physically to \textit{surface tension} and \textit{osmotic pressure}.
\end{itemize}

The sign of $H$ depends on the choice of surface orientation. Additionally, $c_0$ can take different signs in various physical contexts, depending on whether the membrane tends to bend inward or outward in the absence of external forces.

When $c_0 = 0$, the functional simplifies to the classical \textit{Willmore functional}, and its critical points are known as \textit{Willmore surfaces}.

A fundamental step in studying membrane shapes is computing the \textit{first variation} of the functional under normal deformations of a one parameter family of closed surfaces. Since the surface is closed, Green's Theorem can be applied to derive the associated {\it Euler-Lagrange equation}. This equation characterizes the critical points of the Helfrich functional and is commonly referred to in the literature as the \textit{Shape Equation}, \textit{Helfrich Equation}, or \textit{Helfrich Shape Equation} (\cite{Helfrich}, \cite{HuOuYang}, \cite{Vaidya}, \cite{JulicherSeifert}, \cite{Kumar}, \cite{Gonzalez}):

\begin{equation} \label{E-L}
2 \Delta H + (2H - c_0)\left[2H^2 - 2K + c_0 H\right] + \bar{P} - 2\bar{\lambda} H = 0,
\end{equation}
where the normalized parameters are defined by
\[
\bar{P} := \frac{\Delta P}{\beta}, \qquad \bar{\lambda} := \frac{\lambda}{\beta}.
\]

\noindent
These nondimensional parameters are often used in the literature for convenience, except in cases where $\beta = 1$ is assumed.

\begin{rem}
The Helfrich Shape Equation is a nonlinear, fourth-order elliptic partial differential equation which poses substantial analytical challenges. To make the problem more tractable, researchers commonly assume axial symmetry, which allows significant simplification. This axisymmetric model has been widely adopted in the literature.
\end{rem}

\section{Shape Equation for the Particular Case of Axisymmetric Surfaces with Spherical Topology}

Let $X:D \rightarrow \mathbb{R}^3$ be a smooth surface immersion in $\mathbb{R}^3$, where the domain $D$ is defined as $(0, \alpha) \times (0, 2\pi)$, an open, simply connected subset of the plane. We use the following axisymmetric parametrization:
\[
X(s, v) = \left(r(s)\cos v,\, r(s)\sin v,\, z(s)\right),
\]
where $s$ is the arclength along the profile curve $(r(s), z(s))$. Rotating this curve around the $z$-axis generates the surface. Such a surface of revolution is called an \textit{axisymmetric surface}.

We study RBC shapes modeled by closed axisymmetric surfaces, where $s$ denotes arclength, $r$ is the radial coordinate (distance from the $z$-axis), and $\psi$ is the angle between the profile curve's tangent vector and the horizontal axis. This profile can equivalently be described using the function $\psi = \psi(r)$, where $\psi$ is the \textit{tangent angle} at the point with rotation radius $r$. By definition:
\begin{equation} \label{psi}
\psi(r) = -\arctan\left(\frac{dz}{dr}\right).
\end{equation}

The coordinate functions $z(s)$ and $r(s)$ satisfy the differential relationships:
\begin{equation} \label{CC0}
\frac{dz}{ds} = -\sin\psi, \qquad \frac{dr}{ds} = \cos\psi,
\end{equation}
which imply
\[
\frac{dz}{dr} = -\tan\psi.
\]

The signs of these expressions are consistent with our chosen angle orientation: $\psi$ is measured counterclockwise from the horizontal axis, as illustrated in Figure~\ref{The_Axisymmetric_Model}.

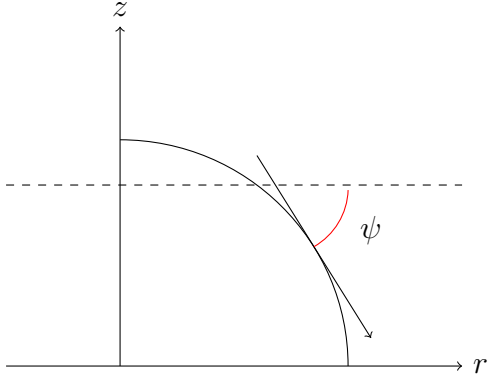
\begin{figure}
    \centering
    \begin{tikzpicture}[scale=3]
        \draw[->] (-0.5,0) -- (1.5,0) node[right] {$r$};
        \draw[->] (0,0) -- (0,1.5) node[above] {$z$};

        \draw[name path=semicircle] (0,1) arc (90:0:1);
        \draw[dashed] (-0.5,0.8) -- (1.5,0.8);

        \pgfmathsetmacro{\slope}{(0.5275 - 0)/(0.8503 - 0)};
        \pgfmathsetmacro{\tangentslope}{-1/\slope};

        \pgfmathsetmacro{\xstart}{0.8503 - 0.25}; 
        \pgfmathsetmacro{\ystart}{0.5275 - \tangentslope * 0.25}; 
        \pgfmathsetmacro{\xend}{0.8503 + 0.25}; 
        \pgfmathsetmacro{\yend}{0.5275 + \tangentslope * 0.25}; 

        \draw[name path=tangent, decoration={markings, mark=at position 1 with {\arrow[scale=1]{>}}}, postaction={decorate}] (\xstart,\ystart) -- (\xend,\yend);

        \begin{scope}[rotate around={300:(0.8503, .5275)}]
            \draw[red] (0.8503, .5275) arc (0:atan(-\tangentslope):0.3); 
        \end{scope}

        \node at (1.1, 0.6) {$\psi$}; 
    \end{tikzpicture}
    \caption{The Axisymmetric Model}
    \label{The_Axisymmetric_Model}
\end{figure}

Using this axisymmetric frame, the curvatures can be re-expressed along the profile curve $z(r)$. Following the formulations in \cite{JulicherSeifert}, we define:

\begin{equation} \label{H}
H = -\frac{k_1 + k_2}{2} = -\frac{1}{2} \left( \cos\psi \frac{d\psi}{dr} + \frac{\sin\psi}{r} \right),
\end{equation}

\begin{equation} \label{K}
K = k_1 k_2 = \cos\psi \sin\psi \cdot \frac{1}{r} \cdot \frac{d\psi}{dr},
\end{equation}
where $\psi(s)$ is again the angle between the profile curve tangent and the horizontal direction.

In our work, we follow \cite{Vaidya} and \cite{JulicherSeifert}, who unified various shape equations appearing in earlier studies. We restrict our attention to axisymmetric surfaces with spherical topology (i.e., genus zero). The shape equation discussed below applies specifically to this setting and does not hold for surfaces of higher genus or for non-axisymmetric cases.

From \cite{Vaidya}, the third-order formulation of the Helfrich Shape Equation is given by:
\begin{align*}
\mathcal{H} = & -\cos^3\psi \frac{d^3\psi}{dr^3} + 4 \sin\psi \cos^2\psi \frac{d^2\psi}{dr^2} - \cos\psi\left( \sin^2\psi - \frac{1}{2}\cos^2\psi \right)\left(\frac{d\psi}{dr} \right)^3 \\
& + \frac{7 \sin\psi \cos^2\psi}{2r} \left( \frac{d\psi}{dr} \right)^2 - \frac{2 \cos^3\psi}{r} \frac{d^2\psi}{dr^2} \\
& + \left( \frac{c_0^2}{2} - \frac{2c_0 \sin\psi}{r} + \frac{\sin^2\psi}{2r^2} + \bar{\lambda} - \frac{\sin^2\psi - \cos^2\psi}{r^2} \right) \cos\psi \frac{d\psi}{dr} \\
& + \bar{P} + \bar{\lambda}\sin\psi - \frac{\sin^2\psi}{2r^3} + \frac{c_0^2\sin\psi}{2r} - \frac{\sin\psi \cos^2\psi}{r^3}.
\end{align*}

This expression can be unwieldy, and several methods exist to simplify it depending on the chosen parameterization. In the case of axisymmetric surfaces with spherical topology, the following simplified form of the shape equation is obtained:
\begin{align}
\mathcal{H} = & \cos^2\psi \frac{d^2\psi}{dr^2} - \frac{\sin\psi \cos\psi}{2} \left(\frac{d\psi}{dr}\right)^2 - \frac{\sin\psi}{2r^2 \cos\psi} - \frac{\sin\psi \cos\psi}{2r^2} \nonumber \\
& - \frac{c_0^2 \sin\psi}{2\cos\psi} + \frac{\cos^2\psi}{r} \frac{d\psi}{dr} - \frac{c_0 \sin^2\psi}{r \cos\psi} - \frac{\bar{P}}{2} \frac{r}{\cos\psi} - \bar{\lambda} \frac{\sin\psi}{\cos\psi} = 0. \label{SE}
\end{align}

For the computational details which precede this shape equation, we refer the reader to \cite{JulicherSeifert} and \cite{Kumar}.

\

\

\section{Cassini Ovals versus Shape Equation}

The Cassini ovals have been investigated in multiple papers, including \cite{Mladenov}. Cassini Ovals are a class of algebraic curves that are associated with toric sections which can be represented in the following form
\begin{equation}
\label{O}
z(r)=\pm\sqrt{\sqrt{4a^2r^2+c^4}-a^2-r^2},
\end{equation}
where $a,c\in\mathbb{R}$. It is convenient to view $a,c$ as a ratio that characterizes the Cassini ovals. We say that $e=\frac{c}{a}$ is the eccentricity and $\epsilon = \frac{a}{c}$ the biconcavity of a Cassini Oval.

Rewriting \eqref{O} in terms of bioconcavity and rescaling we have that, 
\begin{equation}\label{CO}
z(r)=\pm\sqrt{\sqrt{4\epsilon^2r^2+1}-\epsilon^2-r^2}.
\end{equation}
It can be seen that $0\le\epsilon<1$, and therefore the argument under the outer square root
\begin{equation}
\sqrt{4\epsilon^2r^2+1}-\epsilon^2-r^2 \label{sqrt_arg}
\end{equation}
is nonnegative for $0\le r \le\sqrt{1+\epsilon^2}$, and is nonnegative for $\sqrt{-1+\epsilon^2}\le r\le\sqrt{1+\epsilon^2}$ when $\epsilon\ge1$. 
We must then restrict our domain for derivatives to exist, or Eq.~\eqref{sqrt_arg} to be strictly positive.
Letting $\epsilon \geq 0,$ the domain of definition is constrained by
\begin{equation}
D := \left\{ 
\begin{array}{l l l}
0 \le r < \sqrt{1+\epsilon^2} &\mbox{ for } &0\le\epsilon<1\\
\sqrt{-1+\epsilon^2}< r < \sqrt{1+\epsilon^2} &\mbox{ for } &\epsilon\ge1
\end{array}
\right. .
\end{equation}

The main result of the current work is stated and proved below. 

\

\begin{theorem}
    For $\epsilon > 0$, Cassini ovals as a profile curve do not satisfy the shape equation \eqref{SE}. 
\end{theorem}

\begin{proof}
    First, we will assume that the Cassini oval satisfies the shape equation. Referring back, we have that any profile curve must satisfy
    
\begin{align*}
    \mathcal{H}=&\cos^2\psi\frac{d^2\psi}{dr^2}-\frac{\sin\psi \cos\psi}{2}\left(\frac{d\psi}{dr}\right)^2-\frac{\sin\psi}{2r^2\cos\psi}-\frac{\sin\psi \cos\psi}{2r^2}+\frac{\cos^2\psi}{r}\frac{d\psi}{dr}\\
&-\frac{c_0^2\sin\psi}{2\cos\psi}-\frac{c_0\sin^2\psi}{r\cos\psi}-\frac{\bar{P}}{2}\frac{r}{\cos\psi}-\bar{\lambda}\frac{\sin\psi}{\cos\psi}=0.
\end{align*}
\begin{figure}
     \centering
     \includegraphics[width=0.4\linewidth]{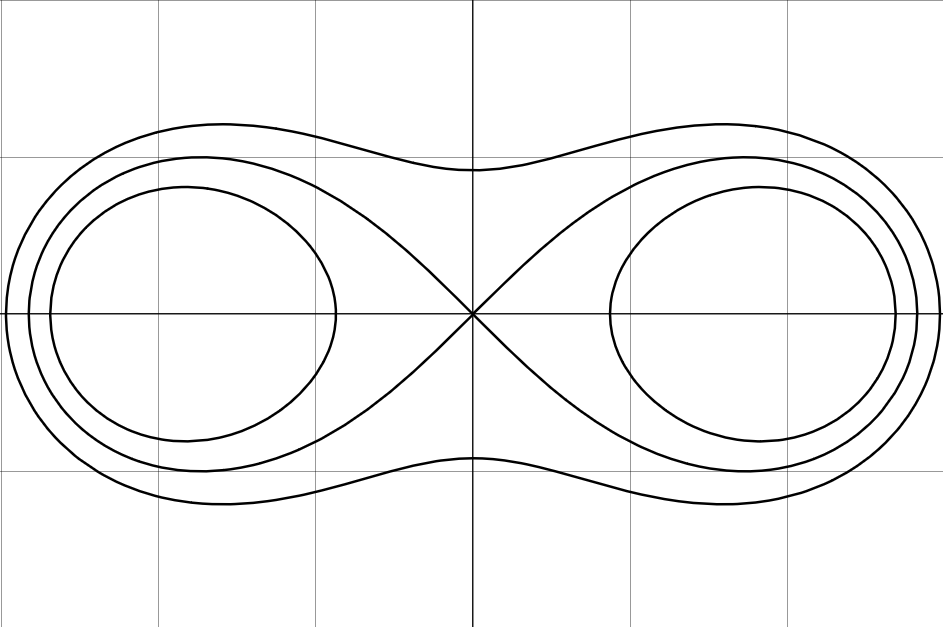}
     \caption{Cassini ovals with e=.9, e=1.0, and e=1.1.}
     \label{fig:Cassini Ovals}
 \end{figure}

We will derive a contradiction by reframing our equation in terms of the Cassini ovals and show they do not satisfy the shape equation. By definition, we have that $\psi(r)=\arctan\left(\frac{dz}{dr}\right).$ Using Eq.~\eqref{CO} and substituting $ u(r) = z'(r)$ yields the following expression
\begin{align*}
\mathcal{H}= &\frac{-5u (u')^2}{2(1+u^2)^3}+\frac{u''}{(1+u^2)^2}-\frac{u}{2r^2}\left(1+\frac{1}{1+u^2}\right) +\frac{1}{r}\frac{u'}{(1+u^2)^2}\\&-\frac{c_0^2}{2}u-\frac{c_0 u^2}{r\sqrt
{1+u^2}}-\frac{\bar{P} r}{2}\sqrt{1+u^2} -\bar{\lambda}u.
\end{align*}
We explicitly compute the first derivative of $z(r),$
\begin{equation}
u(r):=z'(r)=\left[\frac{r}{2\sqrt{\sqrt{1+4 \epsilon ^2r^2}-\epsilon ^2-r^2}}\left(\frac{4\epsilon ^2}{\sqrt{1+4 \epsilon
^2r^2}}-2\right)\right], \label{u_equation}
\end{equation}
and substitute. With the assumption that ${{\sqrt{1+4\epsilon^2r^2}-\epsilon^2-r^2}} > 0$ to simplify our expression we impose 
$$\mathcal{H} \left(\sqrt{1+4 r^2 \epsilon ^2}-\epsilon ^2\right)^3\left(\sqrt{1+4 r^2 \epsilon ^2}-r^2-\epsilon ^2\right)^{3/2} = 0\,, $$
for all $r$. With the aid of the Mathematica tool, substituting and simplifying lead to the following equation
\begin{equation}
\mathcal{H}_1 +\frac{\mathcal{H}_2 }{ \sqrt{1+4 r^2 \epsilon ^2}} + \mathcal{H}_3 \sqrt{{-\epsilon ^2
+\sqrt{1+4 r^2 \epsilon ^2}}} + \mathcal{H}_4  \sqrt{\frac{-\epsilon ^2
+\sqrt{1+4 r^2 \epsilon ^2}}{1+4 r^2 \epsilon ^2}} = 0, \label{simplified_espression}
\end{equation}
with
\begin{align}
\mathcal{H}_1 =& \frac{c_0^2 r}{2}-8 r \epsilon ^2+\frac{13}{2} c_0^2 r^3 \epsilon ^2-6 r^5 \epsilon ^2+7 c_0^2 r \epsilon ^4-72 r^3
\epsilon ^4+18 c_0^2 r^5 \epsilon ^4-10 r^7 \epsilon ^4-24 r \epsilon ^6+\frac{63}{2} c_0^2 r^3 \epsilon ^6\nonumber\\&-114 r^5 \epsilon ^6+\frac{9}{2}
c_0^2 r \epsilon ^8-60 r^3 \epsilon ^8
+r
\bar{\lambda} +13 r^3 \epsilon ^2 \bar{\lambda} +14 r \epsilon ^4 \bar{\lambda} +36 r^5 \epsilon ^4 \bar{\lambda} +63 r^3 \epsilon ^6 \bar{\lambda} +9 r \epsilon ^8 \bar{\lambda},\\ 
\mathcal{H}_2 =& -\frac{1}{2}c_0^2 r^3-{3 c_0^2 r \epsilon ^2}+{14 r^3 \epsilon ^2}-{4 c_0^2 r^5 \epsilon ^2}+{24 r \epsilon
^4}-\frac{57}{2 } {c_0^2 r^3 \epsilon ^4}+{84 r^5 \epsilon ^4}-{8 c_0^2 r^7 \epsilon ^4}-{8 c_0^2 r \epsilon ^6}\nonumber\\&+{158 r^3
\epsilon ^6}-{66 c_0^2 r^5 \epsilon ^6}+{108 r^7 \epsilon ^6}+{8 r \epsilon ^8}-{33 c_0^2 r^3 \epsilon ^8}+{236 r^5
\epsilon ^8}-{c_0^2 r \epsilon ^{10}}+{24 r^3 \epsilon ^{10}}
 -{r^3
\bar{\lambda} }\nonumber\\&-{6 r \epsilon ^2 \bar{\lambda} }-{8 r^5 \epsilon ^2 \bar{\lambda} }-{57 r^3 \epsilon ^4 \bar{\lambda} }-{16 r^7 \epsilon ^4 \bar{\lambda} }-{16
r \epsilon ^6 \bar{\lambda} }-{132 r^5 \epsilon ^6 \bar{\lambda} }-{66 r^3 \epsilon ^8 \bar{\lambda}
}-{2 r \epsilon ^{10} \bar{\lambda} },\\ 
\mathcal{H}_3 =& -c_0 r  +\frac{1}{2} \bar{P} r^3  +2 \bar{P} r \epsilon ^2  -14
c_0 r^3 \epsilon ^2  +2 \bar{P} r^5 \epsilon
^2  -19 c_0 r \epsilon ^4  \nonumber\\& +\frac{19}{2} \bar{P} r^3 \epsilon ^4 -40 c_0 r^5 \epsilon ^4  +2 \bar{P} r \epsilon ^6  -88 c_0 r^3 \epsilon ^6  -16 c_0 r \epsilon ^8 ,\\
\mathcal{H}_4 =& 
-\frac{1}{2} \bar{P} r +c_0 r^3 +7 c_0 r \epsilon ^2 -\frac{11}{2} \bar{P} r^3 \epsilon ^2 +8 c_0 r^5 \epsilon ^2
-3 \bar{P} r \epsilon ^4 +69 c_0 r^3 \epsilon ^4 -14 \bar{P} r^5
\epsilon ^4 \nonumber\\& +16 c_0 r^7 \epsilon ^4 +25 c_0 r \epsilon ^6 -\frac{25}{2}
\bar{P} r^3 \epsilon ^6 +164 c_0 r^5 \epsilon ^6 -\frac{1}{2} \bar{P} r \epsilon ^8 +104 c_0 r^3 \epsilon ^8 +4 c_0 r \epsilon
^{10} .\end{align}
Since Eq.~\eqref{simplified_espression} needs to be satisfied for all $r$ in the domain $D$ we require that
$$ \mathcal{H}_1 = 0, \quad  \mathcal{H}_2 = 0,\quad  \mathcal{H}_3 = 0, \quad \mathcal{H}_4 = 0,$$
which leads to a system of overdetermined algebraic equations that cannot be solved. To see this, it is sufficient to show that the equation $\mathcal{H}_3 = 0$ leads to a contradiction. First, we collect all terms with the same monomial
$$\mathcal{H}_3 = r^5 \left(2 \bar{P} \epsilon ^2-40 c_0 \epsilon ^4\right)+r^3 \left(\frac{\bar{P}}{2}-14 c_0 \epsilon ^2+\frac{19
\bar{P} \epsilon ^4}{2}-88 c_0 \epsilon ^6\right)$$ 
$$+r \left(-c_0+2 \bar{P} \epsilon ^2-19 c_0 \epsilon ^4+2 \bar{P} \epsilon
^6-16 c_0 \epsilon ^8\right) = 0,$$ and since this needs to be satisfied for all $r$ we get the algebraic system 
\begin{equation}
\left\{ 
\begin{array}{l}
2 \bar{P} \epsilon ^2-40 c_0 \epsilon ^4 = 0\\
\dfrac{\bar{P}}{2}-14 c_0 \epsilon ^2+\dfrac{19
\bar{P} \epsilon ^4}{2}-88 c_0 \epsilon ^6 = 0 \\
-c_0+2 \bar{P} \epsilon ^2-19 c_0 \epsilon ^4+2 \bar{P} \epsilon
^6-16 c_0 \epsilon ^8=0
\end{array}
\right. .
\end{equation}
The first equation gives $\bar{P}=20c_0 \epsilon^2.$ 
Substituting and solving in the second and third equations yields 
$\epsilon = \frac{2}{51}^{1/4}$ and $\epsilon = \frac
{1}{2}(\frac{1}{3}(-21+\sqrt
537))^{\frac{1}{4}}$, respectively. This is a contradiction since $\epsilon$ takes on two different values, yielding two different Cassini ovals. It follows then that for $\epsilon> 0 $ Cassini ovals do not satisfy the shape equation.
\end{proof}

\section{Round Spheres as Limiting Case Solutions to the Shape Equation}

\

Though Cassini Ovals as a profile curve are not critical for the Shape Equation, it can be shown that for the limit case of $\epsilon=0$ the Cassini ovals do satisfy \eqref{SE}. The easiest way this can be verified is by considering a sphere $\Sigma$ with radius a.  
Letting $k_1=k_2=H=1/a$ or $k_1=k_2=H=-1/a$, we then substitute directly into the Helfrich Shape Equation \eqref{E-L}.

In \cite{AST}, the authors have shown that $\psi=\arcsin r$ and $\psi=-\arcsin r$ represent solutions that can be verified directly into the axisymmetric shape equation that explicitly depends on $\psi$. More precisely, the following result was proven

\begin{theorem}

I). For $\psi=a\arcsin r,$ in order for spherical solutions to exist, the surface $\Sigma$ must satisfy the following constraint on the parameters $c_0,$ $\bar{P}$, $\bar{\lambda}$:

\begin{equation}
{\bar{P}} a^2+ [c_0^2 +2\bar{\lambda}] a+2c_0 =0
\end{equation}

II). For $\psi=-a\arcsin r,$ in order for spherical solutions to exist, the surface $\Sigma$ must satisfy the following constraint on the parameters $c_0,$ $\bar{P}$, $\bar{\lambda}$:

\begin{equation}
{\bar{P}} a^2 - [c_0^2 +2\bar{\lambda}] a+2c_0 =0.
\end{equation}
    
\end{theorem}

%Thus, we have two symmetric solutions that correspond to different conditions on the pressure, spontaneous curvature and tension. Since $\psi(r)$, $\psi'(r)$ and $\psi''(r)$ will all be changing sign, along with $\sin {\psi}$ changing sign and $\cos {\psi}$ remaining invariant with respect to the change {$\psi \to -\psi$}, we see that the shape equation 
%\begin{equation}
%{\bar{P}} a^2+ [c_0^2 +2\bar{\lambda}] a+2c_0 =0
%\end{equation}
%changes to
%\begin{equation}
%{-\bar{P}} a^2+ [c_0^2 +2\bar{\lambda}] a -2c_0 =0,
%\end{equation}
%which would physically correspond to a sign change for the spontaneous curvature and a sign change in the pressure difference alike. 

  %  For the sake of completion, both solutions must be taken into account, as the spontaneous curvature and pressure may be considered positive or negative and are not a priori prescribed. This is an important consideration, which other studies did not discuss. 
If we impose the condition that $\epsilon=0$ on a Cassini oval, it can be seen that \eqref{CO} corresponds to a circle, whose body of revolution is a round sphere. Based on the previous theorem, we conclude that Cassini ovals trivially satisfy the shape equation for $\epsilon=0.$

\section{Recent RBC Model based on Energy Density}

Recent progress in the analytical and numerical modeling of biconcave discoid shapes, specifically red blood cell (RBC) profiles, has been made by the three authors of \cite{Gonzalez}, including one of the present contributors. This section summarizes these developments and further disseminates their implications for the current study.

Assuming that the bending energy density at every point on the RBC surface is proportional to the average squared deviation of the normal curvature from a reference curvature, denoted by $\langle(\kappa - \mathring{\kappa})^2\rangle$, and that this energy density remains constant over the entire surface, the authors derived a differential equation governing the cross-sectional profile of an axisymmetric RBC. The profile is symmetric with respect to both the rotational axis and the equatorial plane. This differential equation, being quadratic in the first derivative, admits two distinct solution branches.

These two branches jointly define the full RBC profile. The inner region, from the cell center to the inflection point, corresponds to the first solution, which has a monotonically increasing slope starting from zero, ensuring regularity at the center. The outer region, extending from the inflection point to the cell's periphery, corresponds to the second solution, characterized by a monotonically decreasing slope. At the inflection point, the two solutions are found to extends each other regularly.

Two distinct cases are analyzed: one without spontaneous curvature and one incorporating it. The results indicate that only the model including spontaneous curvature can accurately reproduce the observed geometric parameters of real RBCs. In this refined model, the central parts of the RBC (bounded by upper and lower inflection circle) are described as a spherical cap and cup, respectively. Beyond these inflection points, the RBC surface transitions into a toroidal segment, whose cross-section approximates an elliptical arc. This geometric framework allows for accurate approximations of the cell's surface area and enclosed volume.

\begin{figure}
\centerline{\includegraphics[width=10cm]{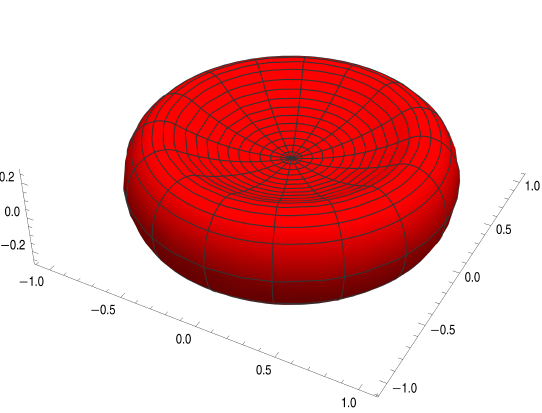}}
\caption{Axisymmetric RBC surface of revolution generated from the analytic profile given by ODE solutions, computed under the condition $H = \mathring{\kappa} \pm a$. \label{3DBloodCell}}
\end{figure}

If the bending energy density is assumed to be a linear combination of the Gaussian curvature and the sum of the squares of the principal curvatures, that is further assumed to be constant, the resulting variational functional becomes equivalent to one with constant mean curvature $H$. By solving the profile differential equation for two different constant values of $H$ (one from the cell center to the inflection point, and the other from the inflection point to the outer edge), the resulting shape will closely match the real RBC geometry. This was demonstrated analytically, numerically and experimentally. 

The two values correspond to $(H - \mathring{\kappa})^2 = a^2$, effectively modeling a constant bending energy density proportional to the square of the adjusted mean curvature.

In this formulation, the inner regions of the RBC surface (near the poles) again form a cup and cap, while the outer region resembles a nodoid surface. In conclusion, models incorporating spontaneous curvature $\mathring{\kappa}$ can successfully replicate the RBC profile. In all cases, two distinct analytic expressions for the vertical coordinate $z$ as a function of the radial coordinate $x$ describe the geometry: one for the region from the center to the inflection point, and one from the inflection point to the periphery.

\section{Conclusion and Open Problems: Case of Non-Constant Spontaneous Curvature}

Besides the limiting case $\epsilon=0,$ there are no straight-forward solutions of $\psi$ that can simultaneously satisfy the shape equations and provide appropriate approximations of the RBC. However, for the cells whose chemical and physical properties varies significantly from point to point, it may be possible to consider non-constant spontaneous curvature and still satisfy \eqref{E-L} and \eqref{SE}.

In general, the spontaneous curvature $c_0$ appearing in the Helfrich-Canham energy is assumed to be constant. This assumption is mathematically convenient and often biologically reasonable when the membrane is homogeneous. However, real biological membranes can exhibit spatial variations in curvature preferences, due to factors such as protein distributions, lipid composition, or interactions with cytoskeletal elements. To capture such heterogeneity, we consider $c_0$ as a function of the radial coordinate $r$, i.e., $c_0 = c_0(r)$.

\vspace{0.3cm}

In the axisymmetric setting, the shape equation involves terms that depend explicitly on $c_0$ and its spatial derivatives. When $c_0$ is non-constant, the equation gains additional terms involving $\frac{dc_0}{dr}$, which can significantly alter the solution landscape. These modifications may lead to new equilibrium shapes that cannot be obtained under the assumption of constant $c_0$.

\vspace{0.3cm}

To illustrate this, consider the shape angle function $\psi(r)$ satisfying the axisymmetric shape equation, now modified to include $c_0(r)$. In this case, we retain the usual definitions of the mean curvature $H$, Gaussian curvature $K$, and the tangent angle $\psi$, but note that the spontaneous curvature term now contributes derivative terms of the form $\partial_r c_0(r)$. As such, the Euler-Lagrange equation becomes more complex and potentially non-integrable in closed form.

\vspace{0.3cm}

Nonetheless, by prescribing specific functional forms for $c_0(r)$, one can obtain analytical or numerical solutions. For instance, the choice
\[
c_0(r) = C_0 \cos^2\left(\frac{\pi r}{2r_0}\right)
\]
has been proposed in several studies to model localized curvature-inducing effects, such as protein adsorption. This profile ensures smoothness and symmetry about the membrane midplane ($r = 0$) and vanishes at the membrane edge ($r = r_0$), consistent with physical constraints.

\vspace{0.3cm}

In this context, we investigate whether such non-constant $c_0(r)$ profiles could be Cassini profiles or other related ones which satisfy the axisymmetric shape equation. This represents a promising direction for modeling red blood cell shapes beyond the classical biconcave disc, especially when additional biophysical constraints are imposed.

\begin{flushleft}
    Eugenio A{\footnotesize ULISA}\\
    Department of Mathematics and Statistics, Texas Tech University, Lubbock, TX, 79409, USA\\
    E-mail: eugenio.aulisa@ttu.edu
\end{flushleft}

\begin{flushleft}
	Stone F{\footnotesize IELDS}\\
	Department of Mathematics and Statistics, Texas Tech University, Lubbock, TX, 79409, USA\\
	E-mail: stofield@ttu.edu
\end{flushleft}

\begin{flushleft}
	Magdalena T{\footnotesize ODA}\\
	Department of Mathematics and Statistics, Texas Tech University, Lubbock, TX, 79409, USA\\
	E-mail: magda.toda@ttu.edu
\end{flushleft}


\begin{thebibliography}{99}


    \bibitem{Vaidya}
    N.~Vaidya, H.~Huang, S.~Takagi, \textit{Correct Equilibrium Shape Equation of Axisymmetric Vesicles,}  
    Constanda, C., Potapenko, S. (eds) Integral Methods in Science and Engineering: Techniques and Applications, Boston, MA: Birkhäuser Boston (2008), 267-276. 
   
    \bibitem{JulicherSeifert}
    F.~J\"ulicher, S.~Udo Seifert, \textit{Shape equations for axisymmetric vesicles: A clarification,} Physical Review E, \textbf{49(5)} (1994), 4728-4731.

    \bibitem{Gruber} A.~Gruber,\textit{ Curvature Functionals and p-Willmore Energy}, Doctoral Dissertation, Texas Tech University. (2019) https://ttu-ir.tdl.org/items/b6f8d68e-a1d3-4044-8299-814a49a004a3

    \bibitem{ABPT} B.~Athukorallage, G.~Bornia, T.~Paragoda, M.~Toda, 
    \textit{Willmore-type energies and Willmore-type surfaces in space forms,}
    Journal of Geometry and Topology, \textbf{18(2)} (2015), 93-108.
    
    \bibitem{Liu} Q.~Liu, Z.~Haijun, J.~Liu, O.~Zhong-Can O., \textit{Spheres and Prolate and Oblate Ellipsoids from an Analytical Solution of Spontaneous Curvature Fluid Membrane Model}, Phys. Rev. E \textbf{60} (1999) 3227-3233.

    \bibitem{Mladenov}B.~Angelov, I.~Mladenov, \textit{On the Geometry of Red Blood Cell,} Geom. Integrability \& Quantization, \textbf{1} (2000), 27-46.

    \bibitem{Pampano}
    B.~Palmer, A.~Pampano, \textit{Stability of Membranes}, ArXiv. (2024), https://arxiv.org/abs/2401.05285.

    \bibitem{Helfrich} W.~Helfrich, \textit{Elastic Properties of Lipid Bilayers: Theory and Possible Experiments,} Zeitschrift für Naturforschung C, \textbf{28(11-12)} (1973), 693-703. 

    \bibitem{HuOuYang} H.~Jian-Guo, O.Y. Zhong-Can, \textit{
    Shape equations of the axisymmetric vesicles,} Physical Review E, 
    \textbf{47(1)} (1993), 461-467. 

    \bibitem{Nitsche} J.C.C. Nitsche,  \textit{Lectures on Minimal Surfaces,} Translated from the German by J.M. Feinberg, Cambridge University Press, Cambridge, \textbf{1} (1989). 

    \bibitem{Deuling} H.J.~Deuling , W.~Helfrich, \textit{Red blood cell shapes as explained on the basis of curvature elasticity}, Biophys J., \textbf{16(8)}(1976), 861-869.
    

    \bibitem{Aulisa-Gruber} E.~Aulisa, A.~ Gruber, \textit{Computational p-Willmore Flow with Conformal Penalty}, {ACM Trans. Graphics}, \textbf{22} (2019), 1-16.

     \bibitem{AST} E.~Aulisa, S.~Fields, M.~Toda, \textit{Red Blood Cells as Helfrich Surfaces with Spherical Topology}, {Proceedings of the International Conference Riemannian Geometry and Applications RIGA 2025, Romanian Journal of Mathematics and Computer Science}, \textbf{15(2)} (2025), 12 pp.

     
    \bibitem{Kumar} G. Kumar, \textit{Shape Equations of the Axisymmetric Vesicles,} Research Project (2022), available at \url{https://home.iitk.ac.in/~kugaurav/projects/Shape_equations_of_the_axisymmetric_vesicles.pdf}.

     \bibitem{Gonzalez} R.C.~Gonzalez,  I.~Mladenov, M.~Toda, \textit{On the Shape of Red Blood Cell.}, {Journal of Geometry and Symmetry in Physics,} \textbf{72} (2025), 1-38.
    
\end{thebibliography}
\end{document}